\documentclass[11pt,reqno]{amsart}

\usepackage[text={155mm,230mm},centering]{geometry}   

\usepackage{graphicx}

\usepackage{epstopdf}
\usepackage{subfigure}
\usepackage{latexsym,bm}

\usepackage{amsmath,amsthm,amssymb,amsfonts}

\usepackage[mathscr]{eucal}

\usepackage[all]{xy}

\usepackage{mathrsfs}

\usepackage{hyperref}

\newcommand{\ac}{\textup{!`}}

\newtheorem{theorem}{Theorem}[section]
\newtheorem{lemma}[theorem]{Lemma}

\newtheorem{proposition}[theorem]{Proposition}

\theoremstyle{definition}

\newtheorem{definition}[theorem]{Definition}
\newtheorem{definition-lemma}[theorem]{Definition-Lemma}
\newtheorem{definition-theorem}[theorem]{Definition-Theorem}

\newtheorem{remark}[theorem]{Remark}

\newtheorem*{ack}{Acknowledgements}
\newtheorem*{notation}{Notations}

\allowdisplaybreaks

\title[Hochschild cohomology of AS-regular algebras]{Koszul duality and the
Hochschild cohomology of Artin-Schelter regular algebras}

\email{liuleilei199009@126.com}

\author{Leilei Liu}

\address{School of Mathematics, Sichuan University, Chengdu, Sichuan Province 610064,
People's Republic of China}


\begin{document}

\begin{abstract}
We identify two Batalin-Vilkovisky algebra structures, one obtained by Kowalzig and Krahmer
on the Hochschild cohomology of an Artin-Schelter regular algebra with semisimple Nakayama 
automorphism and the other obtained by
Lambre, Zhou and Zimmermann
on the Hochschild cohomology of a Frobenius algebra
also with semisimple Nakayama 
automorphism, provided that these two algebras are Koszul
dual to each other.
\end{abstract}

\maketitle

\setcounter{tocdepth}{1}
\tableofcontents

\section{Introduction}

In 2014, Kowalzig and Krahmer showed in
\cite{KK2}
that the Hochschild cohomology
of an Artin-Schelter (AS for short) regular algebra with 
semisimple Nakayama automorphism 
has a Batalin-Vilkovisky algebra
structure.
Soon after that, Lambre, Zhou and Zimmerman proved
in \cite{LZZ} that the Hochschild cohomology of a Frobenius algebra
with semisimple Nakayama automorphism also admits a Batalin-Vilkovisky structure.
In this paper, we identify these two Batalin-Vilkovisky algebra structures, 
provided that these two algebras are Koszul dual to each other.
Let us start with some backgrounds.

In 2008 Brown and Zhang made a progress on the understanding
the twisted Poincar\'e duality of AS-regular algebras.
They proved in \cite{BZ} that for an AS-regular algebra $A$ of global dimension
$n$,
there is an isomorphism between the Hochschild cohomology of $A$
and the Hochschild homology of $A$ with coefficients in $A_\sigma$
\begin{equation}\label{IsotwistedPD}
\mathrm{HH}^\bullet(A)\cong\mathrm{HH}_{n-\bullet}(A; A_\sigma),
\end{equation}
where $\sigma$ is Nakayama automorphism of $A$.
Now if we assume $\sigma$ is {\it semisimple}, 
then Kowalzig and Krahmer proved in \cite{KK2} that
$\mathrm{HH}_\bullet(A; A_\sigma)$ can be computed
by a subcomplex of the corresponding Hochschild complex 
on which the Connes cyclic operator exists.
Therefore we may pull back the Connes cyclic operator to
$\mathrm{HH}^\bullet(A)$ via \eqref{IsotwistedPD}, 
which is usually denoted by $\Delta$.
Kowalzig and Krahmer showed that $\Delta$ in fact genetates
the Gerstenhaber bracket on $\mathrm{HH}^\bullet(A)$,
which means $(\mathrm{HH}^\bullet(A),\cup, \{-,-\}, \Delta)$ is a 
Batalin-Vilkovisky algebra.

Analogously, in 2016, Lambre, Zhou and Zimmermann 
proved in \cite{LZZ} that for a Frobenius algebra 
with semisimple Nakayama automorphism, say $A^!$, 
there also exists a Batalin-Vilkovisky algebra structure on $\mathrm{HH}^\bullet(A^!)$.

To relate these two Batalin-Vilkovisky algebra structures, let us recall a result of P. Smith.
In \cite[Proposition 5.10]{Smith} he showed that
for a graded connected Koszul algebra $A$, $A$ is AS-regular if and only if its Koszul dual 
$A^!$ is Frobenius.
Based on the fact that for a Koszul algebra $A$,
$$\mathrm{HH}^\bullet(A)\cong\mathrm{HH}^\bullet(A^!)$$
as Gerstenhaber algebras,
it is natural to ask that for an AS-regular algebra $A$ with semisimple Nakayama automorphism
whether this above isomorphism is an isomorphism of Batalin-Vilkovisky algebras. In this paper
we give an affirmative answer to this question:

\begin{theorem}\label{mainthm}
Suppose $A$ is a Koszul AS-regular algebra with semisimple Nakayama automorphism.
Denote by $A^!$ its Koszul dual algebra. Then
$$\mathrm{HH}^\bullet(A)\cong\mathrm{HH}^\bullet(A^!)$$
as Batalin-Vilkovisky algebras.
\end{theorem}

This paper may be viewed as a sequel to \cite{CYZ}, 
where the isomorphism of Batalin-Vilkovisky algebras on two Hochschild cohomology
groups are proved for Koszul Calabi-Yau algebras, verifying a conjecture of Rouquier
given in the preprint \cite{Ginzburg} of Ginzburg.

Note that Calabi-Yau algebras and AS-regular algebras are highly related:
in \cite{RRZ}, Reyes, Rogalski and Zhang introduced notion of 
twisted Calabi-Yau algebras (a Calabi-Yau algebra is twisted with trivial twisting), 
and proved that an algebra is twisted Calabi-Yau if and only if it is AS-regular (see also \cite{YeZ}
for some partial result).
Thus the isomorphism of Batalin-Vilkovisky algebras for Koszul
Calabi-Yau algebras, proved in \cite{CYZ},
is a special case of Theorem \ref{mainthm}. In other words,
we may view Theorem \ref{mainthm} as a twisted version of Rouquier's conjecture.

\begin{notation} Throughout this paper,
$k$ denotes a field of character $0$. All tensors and Homs are over $k$ unless otherwise specified.
All algebras (resp. coalgebras) are unital and augmented,
(resp. co-unital and co-augmented) over $k$.
If $A$ is an associative algebra,
then $A^{op}$ is its opposite and
$A^e=A\otimes A^{op}$ is its envelope.
Suppose $V_\bullet$ is a graded vector space, then the shift of the gradings of $V_\bullet$ down
by $n$ is denoted by $s^{-n}V_\bullet$ or $V_\bullet[n]$, i.e., $(V_\bullet[n])_m=V_{m+n}$.
\end{notation}

\section{Preliminaries on Hochschild homology}\label{Sect:Pre}

In this section, we recall the Hochschild homology and cohomology of an associative algebra $A$.
These two homology groups, together with the algebraic operations on them, form
the so-called {\it differential calculus}, a notion introduced by Tamarkin and Tsygan in \cite{TT}.

\subsection{Hochschild homology and cohomology of algebras}

For an associative $k$-algebra $A$, let $\bar A$ be its augmentation ideal.
The {\it reduced Hochschild chain complex} of $A$ with coefficients in an $A$-bimodule $M$,
denoted by $\mathrm{CH}_{\bullet}(A;M)$,
is
$$\cdots\rightarrow M\otimes \bar A^{\otimes n}\xrightarrow{b_n} M\otimes \bar A^{\otimes n-1}
\rightarrow \cdots\rightarrow M\otimes \bar A\xrightarrow{b_1} M\rightarrow 0$$
with boundary $b_n$ given by 
\begin{eqnarray*}
b_n(m, \bar a_1,\cdots, \bar a_n)&=
& (m\bar a_1, \bar a_2,\cdots, \bar a_n)+
\sum_{i=1}^{n-1}(-1)^{i}(m, \bar a_1, \cdots, \bar a_i\bar a_{i+1},\cdots, \bar a_n)\\
&&+ (-1)^n (\bar a_nm, \bar a_1,\cdots,\bar  a_{n-1}),
\end{eqnarray*}
for any $m\in M$ and $\bar a_i\in \bar A$, $i=1,\cdots, n$. 
The associated homology is called the {\it Hochschild homology} of $A$ with coefficients in $M$, 
and is denoted by $\mathrm{HH}_\bullet(A;M)$.

The {\it reduced Hochschild cochain complex} of $A$ with values in $M$,
denoted by $\mathrm{CH}^\bullet(A;M)$, is the complex
$$0\rightarrow M\xrightarrow{\delta_0} \mathrm{Hom}_{k}(\bar A,M)\xrightarrow{\delta_1}
\cdots\rightarrow \mathrm{Hom}_{k}(\bar A^{\otimes n},M)\xrightarrow{\delta_n}\cdots$$
with coboundary $\delta_n$ given by 
\begin{eqnarray*}
(\delta_n f)(\bar a_1, \cdots, \bar  a_{n+1})&=&\bar a_1f(\bar  a_2,\cdots, \bar  a_{n+1})+
\sum_{i=1}^n(-1)^{i}f(\bar  a_1,\cdots, \bar  a_i\bar  a_{i+1},\cdots, \bar  a_{n+1})\\
&&+ (-1)^{n+1}f(\bar  a_1,\cdots, \bar  a_n)\bar  a_{n+1}
\end{eqnarray*}
for any $f\in \mathrm{Hom}(\bar  A^{\otimes n},M)$ and $\bar  a_i\in \bar  A$, $i=1,\cdots, n+1$. 
The associated cohomology is called the {\it Hochschild cohomology} of $A$ with values in $M$, and
is denoted by $\mathrm{HH}^\bullet(A;M)$. 

Later we will use the fact that
$\mathrm{HH}_n(A;M)=\mathrm{Tor}_n^{A^e}(A;M)$ and $\mathrm{HH}^n(A;M)=\mathrm{Ext}^n_{A^e}(A;M)$ ({\it c.f.}
\cite{Weibel}, Lemma 9.1.3). Let us recall the Connes cyclic operator on the Hochschild chain complex.

\begin{definition}[Connes cyclic operator]
Assume $A$ is an associative algebra. The Connes cyclic operator 
$$\mathrm{B}: \mathrm{CH}_n(A;A)\rightarrow \mathrm{CH}_{n+1}(A;A)$$
is given by
$$\mathrm{B}(a_0,\bar a_1,\cdots, \bar a_n)
=\sum_{i=0}^n(-1)^{ni}(1, \bar a_i,\cdots, \bar a_n, \bar a_0,\cdots, \bar a_{i-1}).$$
\end{definition}

It is direct to 
check $\mathrm{B}^2=\mathrm{B}b+b\mathrm{B}=0$, and therefore 
$(\mathrm{CH}_{\bullet}(A;A), b, \mathrm{B})$ is a mixed complex in the sense of Kassel \cite{Kassel}.

\begin{remark}
The Hochschild homology and cohomology
can also be defined for {\it differential graded} algebras.
It is better to view these two homology groups as follows:
Suppose $(A,d_A)$ is a possibly differential graded algebra and $(M,d_M)$ is a differential graded $A$-bimodule.
Let $\mathrm B(A)$ be the bar construction of $A$ (see \cite{Getzler,LV,Weibel} for more details). 
Considering the following total complex $$\mathrm{CH}_\bullet(A; M)\cong M\otimes \mathrm B(A),$$
with the total degree of the tensor product,
and the differential is $b=b_0+b_1$ given by
\begin{eqnarray*}
b_0(m,\bar a_1,\cdots,\bar a_n)&=&-(d_Mm, \bar a_1,\cdots,\bar a_n)-\sum_{i=1}^n(-1)^{\epsilon_{i-1}}(m,\bar a_1,\cdots, d_A\bar a_i,\cdots,\bar a_n)
\end{eqnarray*}
and
\begin{eqnarray*}
b_1(m, \bar a_1,\cdots, \bar a_n)&=& -(-1)^{|m|}(m\bar a_1, \bar a_2,\cdots, \bar a_n)-\sum_{i=1}^{n-1}(-1)^{\epsilon_i}(m, \bar a_1,\cdots, \bar a_i\bar a_{i+1},\cdots \bar a_n)\\
&&+ (-1)^{(|\bar a_n|-1)\epsilon_{n-1}} (\bar a_nm, \bar a_1,\cdots,\bar  a_{n-1}),
\end{eqnarray*}
where $\epsilon_i=|m|+|a_1|+\cdots+|a_i|+i$,
for any homogeneous elements $m\in M$ and $\bar a_i\in \bar A$, $i=1,\cdots n$. Here we denote by $|\bar a_i|$ the degree of $\bar a_i$.
For $M=A$ and taking the gradings into account, the cyclic operator $\mathrm B$ can also be defined in the same way.
Similarly, 
$$\mathrm{CH}^\bullet(A; M)\cong\mathrm{Hom}(\mathrm B(A), M),$$
with the differential on the right-hand side analogously defined.
\end{remark}

\subsection{Differential calculus with duality}

Let us recall the Gerstenhaber cup product and bracket on the Hochschild cohomology of an associative algebra $A$, and its actions on the Hochschild
homology of $A$. Let us start with the notion of Gerstenhaber algebras:

\begin{definition}[Gerstenhaber]
A {\it Gerstenhaber algebra} is a graded $k$-vector space $A^{\bullet}$ 
endowed with two bilinear operators 
$\cup: A^m \otimes A^n\rightarrow A^{m+n}$ and $\{-,-\}:A^n\otimes A^m\rightarrow A^{n+m-1}$ 
such that: 
\begin{enumerate}
\item ($A^{\bullet}, \cup$) is a graded commutative associative algebra, i.e. 
$$a\cup b=(-1)^{|a||b|}b\cup a,$$ with associativity for any homogeneous elements $a,b\in A^\bullet$;
\item ($A^{\bullet},\{-,-\}$) is a graded Lie algebra with the bracket $\{-,-\}$ of degree $-1$, i.e. $$\{a,b\}=(-1)^{(|a|-1)(|b|-1)}\{b,a\}$$ and 
$$\big\{a,\{b,c\}\big\}=\big\{\{a,b\},c\big\}+(-1)^{(|a|-1)(|b|-1)}\big\{b,\{a,c\}\big\},$$ 
or any homogeneous elements $a,b,c\in A^\bullet$;
\item the cup product $\cup$ and the Lie bracket $\{-,-\}$ are compatible in the sense that
$$\{a,b\cup c\}=\{a,b\}\cup c+(-1)^{(|a|-1)|b|}b\cup\{a,c\},$$ 
for any homogeneous elements $a,b,c\in A^\bullet$.
\end{enumerate}
\end{definition}


\begin{definition}[Tamarkin-Tsygan \cite{TT}, Definition 3.2.1]
Let $\mathrm{H}^\bullet$ and $\mathrm{H}_\bullet$ be two graded vector spaces. 
A {\it differential calculus} is a sextuple 
$$(\mathrm{H}^\bullet, \cup, \{-,-\}, \mathrm{H}_\bullet, \mathrm{B}, \cap),$$ 
satisfying the following conditions
\begin{enumerate}
\item $(\mathrm{H}^\bullet, \cup, \{-,-\})$ is a Gerstenhaber algebra;

\item $\mathrm{H}_\bullet$ is a graded module over $(\mathrm{H}^\bullet,\cup)$ by the ``cap action"
$$\cap: \mathrm{H}^n\otimes \mathrm{H}_m \rightarrow \mathrm{H}_{m-n}, f\otimes \alpha \mapsto f\cap \alpha,$$
i.e. $(f\cup g)\cap \alpha=f\cap(g\cap \alpha)$ for any $f\in \mathrm{H}^n$, $g\in \mathrm{H}^m$, $\alpha \in \mathrm{H}_s$;


\item there exists a linear operator $\mathrm B: \mathrm{H}_\bullet\rightarrow \mathrm{H}_{\bullet+1}$
such that $\mathrm B^2=0$ and moreover, if we set
$\mathrm{L}_f(\alpha):=\mathrm B(f\cap \alpha)-(-1)^{|f|}f\cap \mathrm B(\alpha)$,
then 
$$\mathrm{L}_{\{f,g\}}(\alpha)=[\mathrm{L}_f,\mathrm{L}_g](\alpha)$$ 
and 
$$(-1)^{|f|+1}\{f,g\}\cap \alpha =\mathrm{L}_f(g\cap \alpha)-(-1)^{|g|(|f|+1)}g\cap(\mathrm{L}_f(\alpha)).$$
\end{enumerate}
\end{definition}

In the above definition, $\mathrm{L}_f(\alpha)$ is called the Lie derivative
of $f$ on $\alpha$.
%
It is shown by Daletskii, Gelfand and Tsygan in \cite{DGT} that
the Hochschild cohomology and homology of an associative algebra
form a differential calculus. Let us give some details:
\begin{enumerate}
\item The Gerstenhaber 
cup product $\cup:\mathrm{CH}^n(A;A)\otimes \mathrm{CH}^m(A;A)\rightarrow \mathrm{CH}^{n+m}(A;A)$
is given by
      $$f\cup g(\bar a_1 ,\cdots, \bar a_{n+m})
      :=f(\bar a_1 ,\cdots, \bar a_{n})g(\bar a_{n+1},\cdots, \bar a_{n+m}).$$
\item The Gerstenhaber Lie bracket $\{-,-\}:\mathrm{CH}^n(A;A)\otimes 
\mathrm{CH}^m(A;A)\rightarrow \mathrm{CH}^{n+m-1}(A;A)$ is given by
      $$\{f,g\}:=f\circ g-(-1)^{(|f|+1)(|g|+1)}g\circ f,$$
      where 
\begin{multline*}
f\circ g(\bar a_1 ,\cdots,\bar a_{n+m-1})\\
:=\sum_{i=0}^{n-1}(-1)^{(|g|+1)i}f(\bar a_1 ,\cdots,\bar a_i,\overline{g(\bar a_{i+1},\cdots,\bar a_{i+m})},
      \bar a_{i+m+1},\cdots,\bar a_{n+m-1}).
 \end{multline*}

\item The cap product $\cap: \mathrm{CH}^n(A;A)\otimes \mathrm{CH}_m(A;A)\rightarrow \mathrm{CH}_{m-n}(A;A)$ 
is given by
$$f\cap (a_0, \bar a_1 ,\cdots, \bar a_m ):=
(a_0 f(\bar a_1 ,\cdots, \bar a_n ),\bar a_{n+1} , \cdots,\bar a_m).$$

\item The differential operator $\mathrm B$ on $\mathrm{CH}_\bullet(A;A)$ is nothing but the Connes cyclic operator.
\end{enumerate}
One can show that the above operations respect the boundary operators {\it up to homotopy} (which we will not address),
and hence are well-defined on the homology level:

\begin{proposition}[Daletskii-Gelfand-Tsygan \cite{DGT}]
Let $A$ be an associative algebra, then the data 
$(\mathrm{HH}^{\bullet}(A;A),\cup,\{-,-\}, \mathrm{HH}_{\bullet}(A;A),\mathrm{B},\cap)$ forms a differential calculus.
\end{proposition}

\subsection{Another example of differential calculus}\label{Subsect:diffcal2}

Assume $A^!$ is a finite dimensional associative algebra.
Denote $A^{\ac}:=\mathrm{Hom}(A^!; k)$. Then $A^\ac$ has a $A^!$-bimodule structure induced by the natural $A^!$-bimodule structure of $A^!$. There are two operators:
$$\cap^*: \mathrm{CH}^\bullet(A^!)\times \mathrm{CH}^\bullet(A^!;A^{\ac})\rightarrow \mathrm{CH}^\bullet(A^!;A^{\ac})$$
given by
$$(f, \alpha)\mapsto f\cap^*\alpha=(-1)^{|f||\alpha|}\alpha\circ f$$
and
$$
\mathrm B^*: \mathrm{CH}^\bullet(A^!;A^\ac)\rightarrow \mathrm{CH}^\bullet(A^!;A^\ac)
$$
given by
$$
\alpha\mapsto (-1)^{|\alpha|}\alpha\circ \mathrm{B}.
$$
Here $\mathrm{B}$ is the Connes cyclic operator on the Hochschild complex $\mathrm{CH}_\bullet(A^!;A^!)$,
and $\mathrm{CH}^n(A^!;A^\ac)$ is viewed as the linear dual space of the latter via
the following identification
$$\mathrm{CH}^n(A^!;A^\ac)=\mathrm{Hom}_k((A^!)^{\otimes n};A^\ac)\cong \mathrm{Hom}_k((A^!)^{\otimes n+1},k).$$

\begin{theorem}
Assume $A^!$ is a finite dimensional associative algebra.
Denote $A^{\ac}:=\mathrm{Hom}_k(A^!; k)$. 
Then 
$$(\mathrm{HH}^\bullet(A^!),\cup, \{-,-\}, \mathrm{HH}^\bullet(A^!;A^{\ac}), \mathrm{B}^*, \cap^*)$$
is a differential calculus.
\end{theorem}


\begin{proof}
We know $(\mathrm{HH}^\bullet(A^!),\cup, \{-,-\})$ is a Gerstenhaber algebra. 
We need to show the following:

(1) $\mathrm{HH}^\bullet(A^!;A^{\ac})$ is a graded $\mathrm{HH}^\bullet(A^!)$-module. In fact, 
\begin{eqnarray*}
(f\cup g)\cap^*(\alpha)(a_1,\cdots, a_n)&=&(-1)^{(|f|+|g|)|\alpha|}\alpha\circ (f\cup g)\cap(a_1,\cdots, a_n)\\
&=&(-1)^{(|f|+|g|)|\alpha|}\alpha\circ (f\cap (g\cap (a_1,\cdots, a_n)))\\
&=& (-1)^{|g||\alpha|}(f\cap^*\alpha)\circ (g\cap (a_1,\cdots, a_n))\\
&=&(-1)^{|f||g|} g\cap^*(f\cap^*\alpha)(a_1,\cdots, a_n).
\end{eqnarray*}
Moreover, since $\mathrm{HH}^\bullet(A)$ is Gerstenhaber, we have $f\cup g=(-1)^{|f||g|}g\cup f$. 
Thus combining the above two identities, we have
$$(g\cup f)\cap^*(\alpha)=g\cap^*(f\cap^*\alpha),$$
i.e., $\mathrm{HH}^\bullet(A^!;A^{\ac})$ is a graded $\mathrm{HH}^\bullet(A^!)$-module.

(2) Observe that
\begin{eqnarray*}
\mathrm L^*_f(\alpha)&=&\mathrm{B}^*(f\cap^*\alpha)-(-1)^{|f|}f\cap^*(\mathrm{B}^*(\alpha))\\
&=&\mathrm{B}^*((-1)^{|f||\alpha|}\alpha\circ f)-(-1)^{|f|}f\cap^*((-1)^{|\alpha|}\alpha\circ \mathrm{B})\\
&=& (-1)^{|\alpha|+|f|+|f||\alpha|}\alpha\circ\mathrm{B}\circ f-(-1)^{|f|+|\alpha|+|f|(|\alpha|-1)}\alpha\circ f\circ \mathrm{B}\\
&=&(-1)^{|\alpha|(|f|+1)} \alpha\circ \mathrm{L}_f,
\end{eqnarray*}
we have that
\begin{eqnarray*}
\mathrm L^*_{\{f,g\}}(\alpha)&=&\alpha\circ {\mathrm L}_{\{f,g\}}=\alpha\circ\{{\mathrm L}_f,{\mathrm L}_g\}
=\{\mathrm L^*_f,\mathrm L^*_g\}(\alpha)
\end{eqnarray*} 
and moreover,
\begin{eqnarray*}
&&(-1)^{|f|+1}\{f,g\}\cap^*(\alpha)(a_1,\cdots, a_n)\\
&=&(-1)^{|f|+1+(|f|+|g|-1)|\alpha|}\alpha\circ \{f,g\}\cap (a_1,\cdots, a_n)\\
&=&(-1)^{(|f|+|g|-1)|\alpha|}\alpha\circ (\mathrm{L}_f(g\cap (a_1,\cdots, a_n))-(-1)^{|g|(|f|+1)}g\cap(\mathrm{L}_f(a_1,\cdots, a_n)))\\
&=&(-1)^{|g||\alpha|}\mathrm{L}_f((\alpha\circ g)\cap (a_1,\cdots, a_n))-(-1)^{|f||\alpha|-|\alpha|+|g||f|+|g|}g\cap^*\alpha\circ \mathrm{L}_f(a_1,\cdots, a_n)\\
&=&\mathrm L^*_f(g\cap^* \alpha)(a_1,\cdots, a_n)-(-1)^{|g|(|f|+1)}g\cap^*\mathrm L^*_f(\alpha)(a_1,\cdots, a_n)\\
&=&[\mathrm L^*_f,g]\cap^*\alpha(a_1,\cdots, a_n).
\end{eqnarray*}
This completes the proof.
\end{proof}

\subsection{The Batalin-Vilkovisky algebra structure}

In this paper we are more concerned with the Batalin-Vilkovisky algebra structure on
the Hochschild cohomology. It naturally comes from the ``noncommutative Poincar\'e duality"(\cite{VdBP})
of the corresponding algebra. Let us first recall the following notion of
{\it differential calculus with duality}, introduced by Lambre in \cite{Lambre}.

\begin{definition}[Lambre \cite{Lambre}]
A differential calculus $(\mathrm{H}^\bullet, \cup, \{-,-\}, \mathrm{H}_\bullet, \mathrm{B}, \cap)$ 
is called a differential calculus {\it with duality} if there exists an integer $n$ and 
an isomorphism of $\mathrm{H}^\bullet$-modules
$$\phi: \mathrm{H}^{\bullet}\rightarrow \mathrm{H}_{n-\bullet}.$$
\end{definition}

\begin{lemma}[Lambre \cite{Lambre} Theorem 1.6 and Lemma 1.5]\label{calwithdualityinducesBV}
Assume $(\mathrm{H}^\bullet,\cup, \{-,-\}, \mathrm{H}_\bullet,\mathrm B, \cap)$ 
is a differential calculus with duality. 
Let $\Delta:=\phi^{-1}\circ \mathrm{B}\circ \phi$. Then
$$\{a,b\}=(-1)^{|a|+1}(\Delta(a\cup b)-\Delta(a)\cup b-(-)^{|a|}a\cup \Delta(b)).$$
\end{lemma}

\begin{proof}
Denote the imagine $\phi(1)$ of $1$ by $\omega$,
which is also called the volume form of $\mathrm{H}_\bullet$.
Then we have 
$$\phi(\alpha)=\phi(\alpha\cup 1)=\alpha\cap \phi(1)=\alpha\cap \omega,$$
and therefore
\begin{eqnarray*}
&&(-1)^{|a|+1}\phi(\{a,b\})=(-1)^{|a|+1}\{a,b\}\cap \omega =\mathrm L_a(b\cap \omega)-(-1)^{|b|(|a|+1)}b\cap(\mathrm L_a(\omega))\\
&=&\mathrm B(a\cap(b\cap \omega))-(-1)^{|a|}a\cap \mathrm B(b\cap \omega)-(-1)^{|b|(|a|+1)}b\cap(\mathrm B(a\cap \omega)),
\end{eqnarray*}
and similarly,
\begin{eqnarray*}
\phi(\Delta(a\cup b))&=& \mathrm B\circ \phi(a\cup b)=\mathrm B((a\cup b)\cap \omega)=\mathrm B(a\cap (b\cap \omega)),\\
\phi(\Delta(a)\cup b)&=& (-1)^{|b|(|a|+1)}\phi(b\cup \bigtriangleup(a))=(-1)^{|b|(|a|+1)} b\cap \mathrm B(a\cap \omega)
\end{eqnarray*}
and
\begin{eqnarray*}
\phi(a\cup\Delta(b))&=& a\cap (\mathrm B(b\cap \omega)).
\end{eqnarray*}
Hence we have the identity
$$(-1)^{|a|+1}\phi(\{a,b\})=\phi(\Delta(a\cup b)-\Delta(a)\cup b-(-1)^{|a|}a\cup \Delta(b)).$$
Since $\phi$ is an isomorphism, we have
$$(-1)^{|a|+1}\{a,b\}=\Delta(a\cup b)-\Delta(a)\cup b-(-1)^{|a|}a\cup \Delta(b).$$
This completes the proof.
\end{proof}

The above lemma in fact says that $\mathrm{H}^\bullet$ is a Batalin-Vilkovisky algebra.
Let us recall its definition:
\begin{definition}
A {\it Batalin-Vilkovisky algebra} is a 
Gerstenhaber algebra $(\mathrm{H}^\bullet, \cup, \{-,-\})$ together with an operator $\Delta: \mathrm{H}^\bullet\rightarrow \mathrm{H}^{\bullet-1}$ of degree $-1$ such that $\Delta\circ\Delta=0$, $\Delta(1)=0$ and 
$$\{a,b\}=(-1)^{|a|+1}(\Delta(a\cup b)-\Delta(a)\cup b-(-1)^{|a|}a\cup\Delta(b)),$$
for any homogeneous elements $a,b\in \mathrm{H}^\bullet$.
\end{definition}

\section{Artin-Schelter regular algebras}
In this section, we briefly recall the construction
of the Batalin-Vilkovisky algebra on the Hochschild
cohomology of Artin-Schelter regular algebras with
semisimple Nakayama automorphism,
obtained by Kowalzig and Krahmer in \cite{KK2}. In this section, $A$ will present a connected graded algebra over an algebraically closed field $k$. A graded algebra $A$ is said to be connected if $A_i=0$ for $i<0$ and $A_0=k$.

\begin{definition}[Artin-Schelter \cite{AS}]
A connected graded algebra $A$ is called {\it Artin-Schelter regular} (or AS-regular for short) of dimension $d$ if
\begin{enumerate}
\item $A$ has finite global dimension $n$, and
\item $A$ is Gorenstein, that is, 
$\mathrm{Ext}^i_{A}(k,A)=0$ for $i\neq n$ and $\mathrm{Ext}^n_{A}(k,A)\cong k$.
\end{enumerate}
\end{definition}

Later in 2014 Reyes, Rogalski and Zhang proved in \cite{RRZ}
that AS-regular algebras are in fact {\it twisted Calabi-Yau} algebras (see also Yekutieli and Zhang  \cite{YeZ} for some partial results):

\begin{theorem}[\cite{RRZ}, Lemma 1.2]
Suppose $A$ is a connected graded algebra. 
Then $A$ is AS-regular if and only if it is skew Calabi-Yau (in the graded sense), namely, $A$ satisfies
the following two conditions:
\begin{enumerate}
\item $A$ is homologically smooth, that is, 
$A$, viewed as an $A^e$-module, 
has a bounded, finitely generated projective resolution, and
\item there exists an integer $n$ and an algebra automorphism $\sigma$ of $A$ such that
$$\mathrm{Ext}^i_{A^e}(A,A\otimes A)\cong\left\{
\begin{array}{cl}
0, \quad i\neq n,\\
A_{\sigma},\quad i=n
\end{array}\right. $$
as $A^e$-modules.
\end{enumerate}
\end{theorem}

In the above theorem, the automorphism $\sigma$ is called {\it Nakayama automorphism}. 
The $A^e$-module structure of $\mathrm{Ext}^d_{A^e}(A,A\otimes A)$ is induced by 
the inner module structure on $A\otimes A$: $a\cdot (b\otimes c)\cdot d=bd\otimes ac$. 
The module $A_{\sigma}$ is a vector space $A$ equipped with the $A$-bimodule 
structure $a\cdot b\cdot c=ab\sigma(c)$, for any $a,b,c\in A$.
We say $\sigma$ is {\it semisimple} if it is diagonalizable.

In the following, we will always use the above equivalent definition of AS-regular algebras,
rather than its original definition.

\subsection{Results of Kowalzig and Krahmer}

Assume $A$ is an AS-regular algebra with semi-simple Nakayama automorphism $\sigma$. 
In \cite{KK1,KK2}, Kowalzig and Krahmer constructed a differential calculus 
with duality on $(\mathrm{HH}^\bullet(A), \mathrm{HH}_\bullet(A; A_\sigma))$.
Thus as a corollary, they obtained
a Batalin-Vilkovisky algebra structure on $\mathrm{HH}^\bullet(A)$.

First, let us observe that, compared to the differential calculus structure given in \S\ref{Sect:Pre},
there is no Connes operator on $\mathrm{CH}_\bullet(A; A_\sigma)$.
What Kowalzig and Krahmer did is to consider a subcomplex of $\mathrm{CH}_\bullet(A; A_\sigma)$,
whose homology is $\mathrm{HH}_\bullet(A; A_\sigma)$ and
on which the Connes operator is well-defined. Let us briefly recall their results.

Let 
$$\mathrm{B}: \mathrm{CH}_p(A;A_{\sigma})\rightarrow \mathrm{CH}_{p+1}(A;A_{\sigma})$$
be
given by
\begin{eqnarray*}
\mathrm{B}(a_0, a_1,\cdots, a_n)&
:=&\sum_{i=0}^n(-1)^{ni}(1, a_i,\cdots, a_n, a_0, \sigma(a_1),\cdots, \sigma(a_{i-1})).
\end{eqnarray*}
Since $\sigma$ is semisimple,  there is a decomposition of $\mathrm{CH}_\bullet(A;A_{\sigma})$ as follows.
Let $\Lambda$ be the set of eigenvalues of $\sigma$ acting on $A$ and $A_{\lambda}$ be the eigenvalue space corresponding
to $\lambda\in\Lambda$. Denote
$$\mathrm{CH}_\bullet^{\lambda}(A;A_{\sigma})
:=\bigoplus_{i}\bigoplus_{\prod_{j=1}^i\lambda_{i_j}=\lambda}
A_{\lambda_{i_1}}\otimes\cdots\otimes A_{\lambda_{i_i}},\quad \lambda_{i_j}\in\Lambda.$$
The restriction of $b$ makes $\mathrm{CH}_\bullet^{\lambda}(A;A_{\sigma})$
 to be a subcomplex of $\mathrm{CH}_\bullet(A;A_{\sigma})$ and we denote its homology $\mathrm{H}_\bullet(\mathrm{CH}^\lambda_\bullet(A;A_\sigma))$ by $\mathrm{HH}^\lambda_\bullet(A;A_\sigma)$.
A key observation is that, the restriction of $\mathrm B$ on the subcomplex $\mathrm{CH}^1_\bullet(A;A_{\sigma})$ 
 is exactly the Connes operator. Hence $(\mathrm{CH}^1_\bullet(A;A_{\sigma}), b,\mathrm{B})$ 
 is a mixed complex.

Let
$$\mathrm{T}: \mathrm{CH}_p(A;A_{\sigma})\rightarrow \mathrm{CH}_p(A;A_{\sigma})$$
be given by
$$\mathrm{T}(a_0,\cdots, a_n)=(\sigma(a_0),\cdots, \sigma(a_n)).$$

\begin{lemma}[\cite{KK2},(2.19)]
Let $\mathrm B$ and $\mathrm T$ be as above, then there exists 
$$b\mathrm{B}+\mathrm{B}b=\mathrm{Id}-\mathrm{T}$$
on the complex $\mathrm{CH}_\bullet(A;A_{\sigma})$.
\end{lemma}

As an immediate corollary, we have
\begin{equation}\label{equivsubchaincomplex1}
\mathrm{HH}_\bullet(A;A_{\sigma})\cong\mathrm{HH}^1_\bullet(A;A_{\sigma}).
\end{equation}
Via this isomorphism, we obtain the Connes operator $\mathrm B$ on $\mathrm{HH}_\bullet(A; A_\sigma)$.

Similarly, there is a decomposition of the Hochschild cochain complex $\mathrm{CH}^\bullet(A;A)$. 
Let
$$\mathrm{CH}_{\mu}^n(A;A):=\{f\in \mathrm{CH}^n(A;A)|
f(\bar A_{\mu_1}\otimes\cdots\otimes \bar A_{\mu_m})\subset A_{\mu\mu_1\cdots\mu_m}\}.$$
The restriction coboundary $b$ 
makes $\mathrm{CH}^\bullet_{\mu}(A;A)$ to be a subcomplex of 
$\mathrm{CH}^\bullet(A;A)$ and we denote its cohomology $\mathrm{H}^\bullet(\mathrm{CH}^\bullet_{\mu}(A;A))$ 
by $\mathrm{HH}^\bullet_{\mu}(A;A)$. 
In a similar fashion,  Kowalzig and Krahmer proved in \cite{KK1} that the cohomology
is concentrated in the subcomplex corresponding to the eigenvalue 1, namely
\begin{equation}\label{equivsubcochaincomplex1}
\mathrm{HH}^\bullet(A;A)\cong\mathrm{HH}_1^\bullet(A;A).
\end{equation}

It is direct to check that the Gerstenhaber cup product, bracket and cap action restrict to the following maps:
for $\lambda,\mu\in \Lambda$, 
$$
\begin{array}{rccl}\cup:&\mathrm{CH}^p_{\lambda}(A;A)\otimes \mathrm{CH}^q_{\mu}(A;A)&
\rightarrow &\mathrm{CH}^{p+q}_{\lambda\mu}(A;A),\\
\{-,-\}:& \mathrm{CH}^p_{\lambda}(A;A)\otimes \mathrm{CH}^q_{\mu}(A;A)&
\rightarrow &\mathrm{CH}^{p+q-1}_{\lambda\mu}(A;A),\\
\cap:&\mathrm{CH}^{\lambda}_{p}(A;A_{\sigma})\otimes \mathrm{CH}^q_{\mu}(A;A)
&\rightarrow& \mathrm{CH}_{p-q}^{\lambda\mu}(A;A_{\sigma}).
\end{array}
$$
Considering the case of eigenvalue $\lambda=\mu=1$, we have the following theorem.

\begin{theorem}[\cite{KK1}, Theorem 1; \cite{KK2} Theorem 1.5]
Let $\cup_1$, $\cap_1$ and $\{-,-\}_1$ be the restrictions of the cup product, cap product and 
Gerstenhaber bracket to the homology and cohomology spaces associated with the eigenvalue $\lambda=1$. 
Then together with the Connes operator $\mathrm{B}$, they give on
\begin{equation}\label{diffcalonsub1}
(\mathrm{HH}^\bullet_1(A;A),\cup_1, \{-,-\}_1,\mathrm{HH}_\bullet^1(A;A_{\sigma}),\mathrm{B},\cap_1)
\end{equation}
a differential calculus structure.
\end{theorem}

In 2008, Brown and Zhang obtained the following

\begin{theorem}[\cite{BZ}, Corrollary 0.4]\label{PDforASregular}
Suppose $A$ is an AS-regular algebra of finite global dimension $n$. Then we have the following isomorphism
\begin{equation}\label{isoofbz}
\mathrm{HH}_{n-\bullet}(A; A_{\sigma})\cong \mathrm{HH}^\bullet(A;A).
\end{equation}
\end{theorem}


\begin{proof}
Since $A$ is homologically smooth, we have the following
\begin{eqnarray*}
\mathrm{HH}_\bullet(A; A_\sigma)&=&\mathrm{Tor}_\bullet^{A^e}(A, A_\sigma)\\
&=&\mathrm H_\bullet(A\otimes_{A^e}^{\mathrm L} A_\sigma)\\
&=&\mathrm H_\bullet(A\otimes_{A^e}^{\mathrm L}\mathrm{RHom}_{A^e}(A,A\otimes A)[n])\\
&=&\mathrm H_\bullet(\mathrm{RHom}_{A^e}(A, A)[n])\\
&=&\mathrm{Ext}^{n-\bullet}_{A^e}(A, A)\\
&=&\mathrm{HH}^{n-\bullet}(A),
\end{eqnarray*}
which completes the proof.
\end{proof}

Thus combining \eqref{equivsubchaincomplex1}-\eqref{isoofbz} we obtain
the following:

\begin{theorem}[\cite{KK2}, Theorem 4.25]
Suppose $A$ is an AS-regular algebra 
with semisimple Nakayama automorphism $\sigma$. Then
$$(\mathrm{HH}^\bullet(A;A),\cup, \{-,-\},\mathrm{HH}_\bullet(A;A_{\sigma}), \mathrm{B},\cap)$$
forms a differential calculus with duality.
\end{theorem}

Combining the above theorem with Lemma \ref{calwithdualityinducesBV}, we obtain:

\begin{theorem}[\cite{KK2}, Theorem 1.5]
If $A$ is an AS-regular algebra with semisimple Nakayama automorphism, then the Hochschild cohomology $\mathrm{HH}^\bullet(A; A)$ of $A$ is a Batalin-Vilkovisky algebra.
\end{theorem}

\section{Frobenius algebras}
In this section, we rephrase the construction
of the Batalin-Vilkovisky algebra on the Hochschild
cohomology of a Frobenius algebra with
semisimple Nakayama automorphism,
obtained by Lambre, Zhou and Zimmermann in \cite{LZZ}.

\begin{definition}
A graded associative $k$-algebra $A^!$ of finite dimension is called {\it Frobenius} of degree $n$ 
if there exists a nondegenerate bilinear pairing 
\begin{equation}\label{Frobeniuspairing}
\langle -,-\rangle : A^!\otimes A^!\rightarrow  k[n]
\end{equation}
such that 
$\langle ab,c\rangle=\langle a, bc\rangle$,
for all $a,b,c\in A^!$.
\end{definition}

By the nondegeneracy of the pairing, 
there exists an automorphism $\sigma\in \mathrm{Aut}(A^!)$ such that
$\langle ab,c\rangle=(-1)^{|c|(|a|+|b|)}\langle \sigma(c) a,b\rangle$.
Such $\sigma$ is also called Nakayama automorphism of $A^!$. The nondegenerate pairing given by \eqref{Frobeniuspairing}
is equivalent to saying that
$$
\psi: A^!\rightarrow A^{\ac}_\sigma[n],\quad a\mapsto\langle -,a\rangle
$$
is an isomorphism of $A^!$-bimodules.

Let $A^{\ac}=\mathrm{Hom}_k(A^!,k)$ be the linear dual space of $A^!$, which is a graded coalgebra. 
Nakayama automorphism $\sigma$ induces an automorphism $\sigma^*$ on $A^\ac$. 
Here $\sigma^*$ is the adjoint of 
$\sigma$. We have a left co-module structure on $A^{\ac}$:
$$\Delta_l(a)=\sum_{(a)}\sigma^*(a')\otimes a'',$$
for all $a\in A^{\ac}$. (The coproduct on $A^{\ac}$ is viewed as a right co-module structure of $A^{\ac}$,
and is denoted by $\Delta_r(a)=\sum_{(a)}a'\otimes a''$.)
To distinguish, let us denote this new co-bimodule structure 
on $A^{\ac}$ by $\;_{\sigma^*}\! A^{\ac}$.

Recall that in \S\ref{Subsect:diffcal2} we obtain
a differential calculus structure on $(\mathrm{HH}^\bullet(A^!; A^!), \mathrm{HH}_\bullet(A^!; A^{\ac}))$.
In the following we explore this structure in more detail, for $A^!$ being a Frobenius algebra.

\subsection{Hochschild homology of coalgebras}
Suppose $C$ is a coassociative (possibly graded) coalgebra with coproduct $\Delta: C\rightarrow C\otimes C$ given by
$$\Delta(c)=\sum_{(c)}c'\otimes c'':=c'\otimes c''.$$
Let $\Delta^0:=\mathrm{Id}$, $\Delta^1:=\Delta$ and let $\Delta^n:=(\Delta\otimes id\otimes\cdots \otimes id)\circ \Delta^{n-1}$ by recursion.
From the coassociativity of $\Delta$, we have $\Delta^n=(id\otimes\cdots\otimes\Delta\otimes\cdots\otimes id)\circ \Delta^{n-1}$.

\begin{definition}
Suppose $C$ is a coassociative (possibly graded) coalgebra and $M$ is a co-bimodule over $C$.
The Hochschild chain complex of $C$ with coefficients in $M$, denoted by $\mathrm{CH}_\bullet(C; M)$,
is the complex
$$
0\to M\to C\otimes M\to C^{\otimes 2}\otimes M\to\cdots\to
C^{\otimes n}\otimes M\to\cdots
$$
with the boundary $b$ given by
\begin{eqnarray*}
\delta^*(a_1,\cdots, a_n,m)&=&\sum_i(-1)^{|a_1|+\cdots+|a_{i-1}|+i-1+|a_i'|}(a_1,\cdots, a_i', a_i'',\cdots, a_n,m)\\
&+&\sum_{(m)}(-1)^{|a_1|+\cdots +|a_n|+n+|c'|}((a_1,\cdots  , a_n, c', m')+(-1)^{\epsilon}(c'',a_1,\cdots , a_n, m'')),
\end{eqnarray*}
where 
$$\Delta(m)=\sum_{(m)}c'\otimes m'+\sum_{(m)}m''\otimes c''\in C\otimes M\oplus M\otimes C$$
and
$\epsilon=(|c''|-1)(|a_1|+\cdots +|a_n|+n+|m''|)$ for any homogeneous elements $(a_1,\cdots,a_n,m)\in C^{\otimes n}\otimes M$.
The associated homology is called the Hochschild homology of $C$ with coefficients in $M$, and is denoted by
$\mathrm{HH}_\bullet(C; M)$.
\end{definition}

\begin{theorem}[\cite{LZZ}, Proposition 3.3]\label{PDforFrobenius}
Let $A^!$ be a Frobenius algebra of degree $n$ with Nakayama automorphism $\sigma$.
Then there is an isomorphism
$$\mathrm{PD}: \mathrm{HH}^\bullet(A^!;A^!)\cong \mathrm{HH}_{n-\bullet}(A^{\ac};\;_{\sigma^*}\! A^{\ac}).$$
\end{theorem}

\begin{proof}
Given a Frobenius algebra $A^!$ with Nakayama automorphism $\sigma$, 
we have $\psi: A^!\cong A^\ac_\sigma[n]$ as $A^!$-bimodules. 
The $\psi$ is given by $\psi(a)=\langle -,a\rangle$. Therefore we have 
$$\mathrm{Hom}(\mathrm{B}A^!,A^!)\cong \mathrm{Hom}(\mathrm{B}A^!,
A^\ac_{\sigma}[n])\cong \Omega(A^\ac)\otimes A^\ac_{\sigma}[n]\cong \Omega(A^{\ac})\otimes \;_{\sigma^*}\!A^\ac[n].$$
The isomorphisms above are all compatible with boundary maps, and hence we obtain
$$\mathrm{PD}: \mathrm{HH}^\bullet(A^!;A^!)\cong \mathrm{HH}_{n-\bullet}(A^{\ac};\;_{\sigma^*}\! A^{\ac}),$$
which completes the proof.
\end{proof}

\subsection{Frobenius algebra with semisimple Nakayama automorphism}

In this subsection, we go over the Batalin-Vilkovisky structure on the Hochschild cohomology of a Frobenius algebra with semisimple Nakayama automropshim.

Firstly, we consider the Hochschild chain complex $\mathrm{CH}_\bullet(A^{\ac};\;_{\sigma^*}\!A^{\ac})$ of $A^\ac$ with coefficients in $\;_{\sigma^*}\!A^\ac$. Similar to the AS-regular case, there is a decomposition on the chain complex of a Frobenius coalgebra $A^{\ac}$.

Since $\sigma$ is semisimple, there is a decomposition of $\mathrm{CH}_\bullet(A^{\ac};\;_{\sigma^*}\!A^{\ac})$ as follows. Let $\Lambda$ be the set of eigenvalues of $\sigma^*$ and $A^{\ac}_{\lambda}$ be the eigenvalue space corresponding
to $\lambda\in\Lambda$. Let
$$\mathrm{CH}^\lambda_\bullet(A^{\ac};\;_{\sigma^*}\!A^{\ac})
:=\bigoplus_{i}\bigoplus_{\prod_{j=1}^i\lambda_{i_j}=\lambda}
A^{\ac}_{\lambda_{i_1}}\otimes\cdots\otimes A^{\ac}_{\lambda_{i_i}},\quad \lambda_{i_j}\in\Lambda.$$
The restriction of $d$ makes $\mathrm{CH}_\bullet^{\lambda}(A^{\ac};\;_{\sigma^*}\!A^{\ac})$
 to be a subcomplex of $\mathrm{CH}_\bullet(A^{\ac};\;_{\sigma^*}\! A^{\ac})$ and we denote its homology $\mathrm{H}_\bullet(\mathrm{CH}^\lambda_\bullet(A^{\ac};\;_{\sigma^*}\! A^{\ac}))$ by $\mathrm{HH}^\lambda_\bullet(A^{\ac};\;_{\sigma^*}\! A^{\ac})$.

Secondly, we define a map
$$\mathrm{B}:\mathrm{CH}_n(A^{\ac};\;_{\sigma^*}\! A^{\ac})\rightarrow \mathrm{CH}_{n-1}(A^{\ac};\;_{\sigma^*}\! A^{\ac}),$$
given by
$$(a_1,\cdots, a_n, a_0)\mapsto \sum_i(-1)^{i(n-i)}\epsilon(a_0)(a_{i+1},\cdots, a_n, \sigma^*(a_1),\cdots, \sigma^*(a_{i-1}), a_i),$$
where $\epsilon(a_0)$ is the imagine of the counit map $\epsilon: A^\ac\rightarrow k$.
The restriction map $\mathrm B$ on the subcomplex $\mathrm{CH}^1_\bullet(A^{\ac};\;_\sigma\! A^{\ac})$ 
 is exactly the Connes operator. Hence $(\mathrm{CH}^1_\bullet(A^{\ac};\;_{\sigma^*}\! A^{\ac}), b,\mathrm{B})$ is a mixed complex. 

Let
$$\mathrm{T}:\mathrm{CH}_n(A^{\ac};\;_{\sigma^*}\! A^{\ac})\rightarrow \mathrm{CH}_n(A^{\ac};\;_{\sigma^*}\! A^{\ac})$$
be
$$(a_1,\cdots, a_n, a_0)\mapsto (\sigma^*(a_1),\cdots, \sigma^*(a_n),\sigma^*(a_0)).$$
Then we have the following

\begin{lemma}[\cite{LZZ}, Proposition 2.1, \cite{KK2},(2.19)]
On the space $\mathrm{CH}_\bullet(A^{\ac};\;_{\sigma^*}\!A^{\ac})$, there exists the identity
$$b\circ \mathrm{B}+\mathrm{B}\circ b=\mathrm{Id}-\mathrm{T}.$$
\end{lemma}

The above lemma implies that 
$$\mathrm{HH}_\bullet(A^{\ac};\;_{\sigma^*}\! A^{\ac})\cong \mathrm{HH}^1_\bullet(A^{\ac};\;_{\sigma^*}\! A^{\ac}).$$
Similarly, there is a decomposition of the Hochschild cochain complex $\mathrm{CH}^\bullet(A^!;A^!)$. 
Let
$$\mathrm{CH}_{\mu}^n(A^!;A^!):=\{f\in \mathrm{CH}^n(A^!;A^!)|
f(\bar A^!_{\mu_1}\otimes\cdots\otimes \bar A^!_{\mu_m})\subset A^!_{\mu\mu_1\cdots\mu_m}\}.$$
The restriction coboundary $b$ 
makes $\mathrm{CH}^\bullet_{\mu}(A^!;A^!)$ into be a subcomplex of 
$\mathrm{CH}^\bullet(A^!;A^!)$ and we denote its homology $\mathrm{H}^\bullet(\mathrm{CH}^\bullet_{\mu}(A^!;A^!))$ 
by $\mathrm{HH}^\bullet_{\mu}(A^!;A^!)$. Similarly to \eqref{equivsubcochaincomplex1}
we have
\begin{equation}\label{equivsubcochaincomplex2}
\mathrm{HH}^\bullet(A^!;A^!)\cong\mathrm{HH}_1^\bullet(A^!;A^!).
\end{equation}

It is direct to check that the Gerstenhaber cup product, bracket and cap product restrict to the following maps:
for $\lambda,\mu\in \Lambda$, 
$$
\begin{array}{rccl}\cup:&\mathrm{CH}^p_{\lambda}(A^!;A^!)\otimes \mathrm{CH}^q_{\mu}(A^!;A^!)&
\rightarrow &\mathrm{CH}^{p+q}_{\lambda\mu}(A^!;A^!),\\
\{-,-\}:& \mathrm{CH}^p_{\lambda}(A^!;A^!)\otimes \mathrm{CH}^q_{\mu}(A^!;A^!)&
\rightarrow &\mathrm{CH}^{p+q-1}_{\lambda\mu}(A^!;A^!),\\
\cap:&\mathrm{CH}^{\lambda}_{p}(A^{\ac};\;_{\sigma^*}\!A^{\ac})\otimes \mathrm{CH}^q_{\mu}(A^!;A^!)
&\rightarrow& \mathrm{CH}_{p-q}^{\lambda\mu}(A^{\ac};\;_{\sigma^*}\!A^{\ac}).
\end{array}
$$

Consider the case of eigenvalue $\lambda=\mu=1$; we have the following theorem.

\begin{theorem}[\cite{LZZ}, Theorem 2.3]
Let $\cup_1$, $\cap_1$ and $\{-,-\}_1$ be the restrictions of the cup product, cap product and 
Gerstenhaber bracket to the homology and cohomology spaces associated with the eigenvalue $\lambda=1$. 
Then the Connes operator $\mathrm{B}$ gives on
\begin{equation}\label{diffcalonsub2}
(\mathrm{HH}^\bullet_1(A^!;A^!),\cup_1, \{-,-\}_1,\mathrm{HH}_\bullet^1(A^{\ac};\;_{\sigma^*}\!A^{\ac}), \mathrm{B}, \cap_1)
\end{equation}
a differential calculus structure.
\end{theorem}
Together with Theorem \ref{PDforFrobenius}, we obtain the following.
\begin{theorem}[\cite{LZZ}, Theorem 2.3 and Proposition 3.4]
Suppose $A^!$ is a Frobenius algebra of dimension $d$
with semisimple Nakayama automorphism. Then
$$(\mathrm{HH}^\bullet(A^!;A^!),\cup, \{-,-\},\mathrm{HH}_\bullet(A^{\ac};\;_{\sigma^*}\!A^{\ac}),\mathrm{B}, \cap)$$
forms a differential calculus with duality.
\end{theorem}

Combining the above theorem with Lemma \ref{calwithdualityinducesBV}, we obtain:

\begin{theorem}[\cite{LZZ}, Theorem 4.1]
If $A^!$ is a Frobenius algebra with semisimple Nakayama automorphism, then the Hochschild cohomology $\mathrm{HH}^\bullet(A^!; A^!)$ of $A$ is a Batalin-Vilkovisky algebra.
\end{theorem}

\section{Koszul duality of AS-regular algebras}

In this section, we study Koszul AS-regular algebras, and then relate the two differential calculus structures in 
previous two sections by means of Koszul duality. We begin with Koszul algebras, which was introduced by Priddy in \cite{Priddy}.

Assume $V$ is a $k$-vector space generated by a basis $\{x_i\}_{i=1}^n$ of degree $1$. 
The free algebra generated by $V$ is denoted by $\mathrm T(V)$. Let $R\subset V\otimes V$ be a subspace, and
let $(R)$ be the bi-sided ideal generated by $R$ in $\mathrm T(V)$. 
The quotient algebra $A=\mathrm T(V)/(R)$ is called a {\it quadratic algebra}. 

\begin{definition}
Given a quadratic algebra $A=\mathrm T(V)/(R)$, the linear dual of $V$ is denoted by 
$V^*:=\mathrm{Hom}_k(V,k)$ and let $R^{\perp}:=\{f\in V^*\otimes V^*, f(r)=0, \forall r\in R \}$. 
Then $A^!:=\mathrm T(V^*)/(R^{\perp})$ is called the {\it quadratic dual} of $A$. 
\end{definition}
 
Let $A^\ac:=\mathrm{Hom}_k(A^!,k)$, then 
$$A^\ac_n\cong\bigcap^n_{i+j+2=n}(sV)^{\otimes i}\otimes (s^2R)\otimes (sV)^{\otimes j},$$
which is a coalgebra. Its coproduct is the restriction of the coproduct $\Delta$ on the co-free coalgebra 
$\mathrm{T}^c(sV)$ given by
$$\Delta(a_1, \cdots, a_n)=\sum_{i=0}^n(-1)^{i}(a_1, \cdots, a_i)\otimes (a_{i+1},\cdots, a_n).$$
Here the summand for
$i=0$ is $1\otimes (a_1,\cdots, a_n)$ and the summand for $i=n$ is $(a_1,\cdots, a_n)\otimes 1$.

The {\it Koszul complex} associated to $A$ is the complex
\begin{equation}\label{Koszulcomplex}
0\rightarrow A\otimes A^\ac_n
\xrightarrow{b} A\otimes A^\ac_{n-1}\rightarrow \cdots\rightarrow A\otimes A^\ac_1\xrightarrow{b} A\rightarrow k\rightarrow 0,
\end{equation}
with the differential $b$ given by 
$$b(a\otimes f)=\sum_i ax_i\otimes f x_i^*,$$
where $\{x_i^*\}_{i=1}^n$ is the dual basis of $V$ in $V^*$.
It is direct to check $b^2=0$.

\begin{definition}
A quadratic algebra $A$ is called {\it Koszul} if the complex \eqref{Koszulcomplex} is exact. 
In this case, $A^!$ is 
called the {\it Koszul dual algebra} of $A$, and
$A^\ac$ is called the {\it Koszul dual coalgebra} of $A$.
\end{definition}

One of the advantages of Koszul algebras is that $A$ has a much smaller free resolution,
which is described as follows. 
Recall that the cobar construction $\Omega(A^\ac)$ of $A^\ac$ is the free tensor differential graded algebra generated by 
$s^{-1}\bar A^\ac$ with the differential $d$ given by
\begin{eqnarray*}
&&d_{\Omega(A^\ac)}(s^{-1}a_1,\cdots, s^{-1}a_n)\\
&=&\sum_{i=1}^n(-1)^{|a_1|+\cdots+|a_{i-1}|+i-1+|a_i'|}(s^{-1}a_1,\cdots,s^{-1}a_{i-1},s^{-1}a_i',s^{-1}a_i'',s^{-1}a_{i+1},\cdots, s^{-1}a_n)
\end{eqnarray*}
for any $a_i\in \bar A^\ac$, $i=1,\cdots, n$.

Consider the composition of the following maps
$$(A^\ac)^{\otimes n}\xrightarrow{(p)^{\otimes n}}V^{\otimes n}\longrightarrow A_n,$$
where $p: A^\ac\rightarrow A^\ac_1=V$ is the projection map and $V^{\otimes n}\rightarrow A_n$ is the natural surjective map.
The composition map $q$ is denoted by
$$q:\Omega(A^\ac)\rightarrow A.$$
For any $a_k\in A^\ac$, $k=1,\cdots, n$, let
$$q(a_1,\cdots, a_n)=\bar a_1\cdots \bar a_n,$$
where $\bar a_i$ is the image of the projection $p$.  

\begin{proposition}[{\it c.f.} \cite{LV} Theorem 3.4.4]
Suppose $A$ is a Koszul algebra. Then 
$$q:\Omega(A^\ac)\rightarrow A$$
is an quasi-isomorphism.
\end{proposition}

Similarly, recall that $A^\ac_m$ is a subset of $V^{\otimes m}$. 
Let $q':A^\ac_m \rightarrow  \bar A^{\otimes m}$ 
be the restriction map of the natural inclusion $V^{\otimes m}\rightarrow\bar A^{\otimes m}$,
which extends to be a differential graded(DG) coalgebra map $\tilde q: 
A^{\ac}\to \mathrm B(A)$. Then $\tilde q$ is also a quasi-isomorphism.

\subsection{Homology of Koszul algebras with algebraic automorphisms}
Suppose $A$ is a Koszul algebra of global dimension $n$. Let $\sigma$ be an algebra automorphism 
of $A$ preserving the grading.
Since $A_1=V$, we have $\sigma(V)=V$. Extending $\sigma$ to be an algebra map on $\mathrm T(V)$, we thus have $\sigma(R)=R$.
This also means $\sigma$, by restriction on $A^{\ac}$, is an automorphism of vector spaces.

\begin{lemma}
$\sigma$ is a coalgebra automorphism of $A^\ac$.
\end{lemma}
\begin{proof}
Recall that $A^\ac_n\cong\bigcap_{i+j+2=n}V^{\otimes i}\otimes R\otimes V^{\otimes j}$.
Now we prove
$$(\sigma\otimes\sigma)\circ\Delta=\Delta\circ\sigma.$$
For any homogeneous element $r\in A^\ac_m$, it has the form $v_1\otimes\cdots\otimes v_m$. So
$$\sigma(r)=\sigma(v_1)\otimes\cdots\otimes \sigma(v_m)\in A^\ac_m.$$
We have
\begin{eqnarray*}
(\sigma\otimes\sigma)\circ\Delta(r)&=&(\sigma\otimes\sigma)(\sum_k (-1)^k (v_1\otimes \cdots\otimes v_k)\otimes (v_{k+1}\otimes\cdots\otimes v_m))\\
&=&\sum_k (-1)^k(\sigma(v_1)\otimes \cdots\otimes \sigma(v_k))\otimes (\sigma(v_{k+1})\otimes\cdots\otimes \sigma(v_m))
\end{eqnarray*}
and
\begin{eqnarray*}
\Delta\circ\sigma(r)&=&\Delta(\sigma(v_1)\otimes\cdots\otimes \sigma(v_m)\\
&=&\sum_k(-1)^k (\sigma(v_1)\otimes \cdots\otimes \sigma(v_k))\otimes (\sigma(v_{k+1})\otimes\cdots\otimes \sigma(v_m)).
\end{eqnarray*}
This implies that $\sigma$ is a coalgebra map.
\end{proof}

Consider the following complex 
\[
0\rightarrow A_{\sigma}\otimes A^\ac_n\xrightarrow{b}  A_{\sigma}\otimes A^\ac_{n-1}\rightarrow\cdots\rightarrow  A_{\sigma}\otimes A^\ac_1\xrightarrow{b}A\rightarrow 0
\]
with differential $b$ given by
\[
b(a\otimes f)=a\sigma(x_i)\otimes x_i^*f+(-1)^mx_ia\otimes fx_i^*,
\]
for any $a\otimes f\in A_{\sigma}\otimes A^\ac_m$.
It is direct to check $b^2=0$ by the following
\begin{eqnarray*}
b^2(a\otimes f)&=&b(a\sigma(x_i)\otimes x_i^*f+(-1)^mx_ia\otimes fx_i^*)\\
&=&a\sigma(x_i)\sigma(x_j)\otimes x_j^*x_i^*f+(-1)^{m-1}x_ja\sigma(x_i)\otimes x_i^*fx_j^*\\
&&+(-1)^mx_ia\sigma(x_j)\otimes x_j^*fx_i^*+(-1)^{m+(m-1)}x_jx_ia\otimes fx_i^*x_j^*\\
&=&0.
\end{eqnarray*}
We denote this complex by $\mathrm{K}_\bullet(A_{\sigma})=(A_{\sigma}\otimes A^\ac,b)$.

\begin{proposition}There is a quasi-isomorphism
$$q: \mathrm{K}_\bullet(A_{\sigma})\rightarrow (\mathrm{CH}_\bullet(A;A_{\sigma}),b).$$
\end{proposition}

\begin{proof}
Consider the complex  $\mathrm{K}'_\bullet(A)=\bigoplus_m A\otimes A^\ac_m\otimes A$
with the differential $b'=b'_L+b'_R$, 
where $b'_L$ and $b'_R$ are given by
$$b'_L(r\otimes f\otimes s)=\sum_i rx_i\otimes fx_i^*\otimes s,
\quad b'_R(r\otimes f\otimes s)=\sum_i (-1)^mr\otimes x_i^*f\otimes x_is,$$
for any $r\otimes f\otimes s\in \mathrm{K}_m(A)$.
It is direct to check that 
$(b'_L)^2=(b'_R)^2=b'_Lb'_R+b'_Rb'_L=0$. 
Hence $b^2=0$. 
The Koszul property of $A$ (\textit{c.f.} \cite{Krahmer}, Proposition 19) implies that $\mathrm{K}'_\bullet(A)$ is 
a resolution of $A$ as $A$-bimodules. 

Now we have two $A$-bimodules free resolutions of $A$, $A\otimes A^\ac\otimes A$ as above and
the two sided bar resolution $\tilde{\mathrm{B}} (A)$ (recall that it is $A\otimes \mathrm B(A)\otimes A$ 
with extra twisted differential). Recall that $A^\ac_m$ is a subset of $V^{\otimes m}$. 
Let $q': A\otimes A^\ac\otimes A\rightarrow A\otimes \mathrm{B} (A)\otimes A$
be the extension of $\tilde q$, which then commutes with the differentials on both sides.
Then $q'$ is a quasi-isomorphism (see \cite{VdBT}, Proposition 3.3).

It is direct to see that $A_{\sigma}\otimes_{A^e}\mathrm{K}'_\bullet(A)
=(\mathrm{K}_\bullet(A_{\sigma}),b)$,
and $A_{\sigma}\otimes_{A^e}\tilde{\mathrm{B}}(A)=(\mathrm{CH}_\bullet(A; A_{\sigma}),b)$. 
Let $q=\mathrm{Id}\otimes q'$. Then $q$ is the desired quasi-isomorphism.
\end{proof}

\subsection{Two quasi-isomorphisms}
Suppose $A$ is a Koszul AS-regular algebra.
The following result is nowadays well-known.

\begin{theorem}[Smith \cite{Smith}, Proposition 5.10]
Suppose $A$ is a Koszul algebra. Then $A$ is AS-regular if and only if $A^!$ is Frobenius.
\end{theorem}

Now suppose $A$ admits semisimple Nakayama automorphism $\sigma$. 
By Van den Bergh (see \cite{VdBN}, Theorem 9.2), the adjoint $\sigma^*$ of $\sigma$ is the Nakayama automorphism of $A^!$. 
Since $\sigma$ is semisimple, $\sigma^*$ is also semisimple.
Recall from previous subsection that $\sigma$ also gives semisimple Nakayama automorphism of
$A^{\ac}$.
The purpose of this subsection is to prove the following
$$
\mathrm{HH}_\bullet(A; A_\sigma)\cong\mathrm{HH}^\bullet(A^!;\;_\sigma\!A^{\ac}),
$$ 
which commutes with $\mathrm B$ on the left and $\mathrm B^*$ on the right.

Recall that the cobar construction $\Omega(A^{\ac})$ is a DG free algebra, we may extend
the coalgebra automorphism
$\sigma: A^{\ac}\to A^{\ac}$ to be a DG algebra automorphism of $\Omega(A^{\ac})$.
Consider the complex $\mathrm{CH}_\bullet(\Omega(A^{\ac}); \Omega(A^{\ac})_{\sigma})$
(sometimes also denoted by $\Omega(A^\ac)_\sigma\otimes \mathrm{B}\Omega(A^\ac)$).
Assume $\sigma$ is semisimple, then we set $A_{\mu}^{\ac}:=\{a\in A^{\ac}|\sigma(a)=\mu a\}$. Let us denote
$$\Omega(A^{\ac})_{\mu}=\bigoplus_{n\geq 0}\bigoplus_{\Pi_{i=1}^n\mu_i=\mu}A^{\ac}_{\mu_1}\otimes\cdots\otimes A^{\ac}_{\mu_n}$$ 
and 
$$\mathrm{CH}_\bullet^{\mu}(\Omega(A^{\ac}); \Omega(A^{\ac})_{\sigma})=\bigoplus_{n\geq 0}\bigoplus_{\Pi_{i=0}^n\mu_i=\mu} \Omega(A^{\ac})_{\mu_0}\otimes \Omega(A^{\ac})_{\mu_1}\otimes \cdots\otimes \Omega(A^{\ac})_{\mu_n}.$$
The restriction map of $d$ makes $\mathrm{CH}_\bullet(\Omega(A^{\ac}); \Omega(A^{\ac})_{\sigma})_{\mu}$ to be a complex. 
Denote its homology group by $\mathrm{HH}^\mu_\bullet(\Omega(A^{\ac}),\Omega(A^{\ac})_\sigma)$.
Again by Kowlzig and Krahmer (\cite{KK2}, Proposition 2.7, Lemma 7.1) we have the following:

\begin{lemma}
On the complex $$\mathrm{CH}_\bullet(\Omega(A^{\ac}); \Omega(A^{\ac})_{\sigma}),$$ we have the identity
$$d\circ \mathrm{B}+\mathrm{B}\circ d=\mathrm{Id}-\mathrm{T}.$$
\end{lemma}

The above lemma implies 
$$\mathrm{HH}_\bullet(\Omega(A^{\ac}); \Omega(A^{\ac})_{\sigma})\cong\mathrm{HH}^1_\bullet(\Omega(A^{\ac}); \Omega(A^{\ac})_{\sigma}),$$
and the sub complex $\mathrm{CH}^1_\bullet(\Omega(A^{\ac}); \Omega(A^{\ac})_{\sigma})$ is a mixed complex.
Now we have the following two lemmas.

\begin{lemma}
Let $A$ be a Koszul algebra with an semi-simple automorphism $\sigma$, and $A^{\ac}$ be its Koszul dual coalgebra. Then we have a commutative diagram of quasi-isomorphisms of complexes up to homotopy
$$
\xymatrix{
\mathrm{CH}^1_\bullet(A;A_\sigma)\ar[d]_{i} &\mathrm{CH}^1_\bullet(\Omega(A^{\ac}); \Omega(A^{\ac})_\sigma)\ar[l]_-{p_1}& \mathrm{CH}^1_\bullet(A^{\ac};\;_\sigma\! A^{\ac})\ar[l]_-{p_2}\\
\mathrm{CH}_\bullet(A;A_\sigma)& A_{\sigma}\otimes A^{\ac}\ar[l]^{\phi_1} & \mathrm{CH}_\bullet(A^{\ac};\;_\sigma\! A^{\ac})\ar[u]_{p}\ar[l]^{\phi_2}.
}
$$
\end{lemma}
\begin{proof}
The map $p_1$ is defined by
$$p_1: ((v_1\cdots v_n), a_1,\cdots, a_m)\mapsto ((\bar v_1\cdots\bar v_n), \bar a_1,\cdots, \bar a_m),$$
here $a_i$ has the form $a_i:=(u_i^1\cdots u_i^{m_i})$  with $a_i^s\in A^\ac$, $s=1,\cdots, m_i$. And $\bar a_i$ is the imagine of $p:\Omega(A^\ac)\rightarrow A$, that is, $p(u_i^1\cdots u_i^{m_i})=\bar u_i^1\cdots \bar u_i^m\in A_m$.  The map $p_2$ is given by
$$p_2: ((v_1\cdots v_n), u)\mapsto ((v_1\cdots v_n), (u+\Delta(u)+\cdots+\Delta^n(u)+\cdots)),$$
for $v_i\in A^\ac$, $i=0,\cdots, n$. At the bottom of the diagram, the map $\phi_1$ is given by
$$\phi_1: (a, (v_1\cdots v_m))\mapsto (a, v_1,\cdots, v_m),$$
for $a\in A$ and $v_i\in V$. And the map $\phi_2$ is given by
$$\phi_2: (v_1,\cdots, v_n, u)\mapsto (\bar v_1\cdots \bar v_n, u),$$
for $v_i\in A^\ac$, $i=0,\cdots, n$. In the vertical direction, $i$ is the injective map and $p$ is the projective map. They are quasi-isomophisms up to homotopy. All these maps are all morphisms of complexes.

By a spectral sequence argument, 
all these morphisms are quasi-isomorphic. For example,
let us consider $p_1$. There exist filtrations on these two complexes given by
$$ \mathrm F_i(\mathrm{CH}^1_\bullet(\Omega(A^{\ac});\Omega(A^\ac)_\sigma ))=\bigoplus_{j\leq i}\{(a_0,a_1,\cdots, a_j)|a_k\in \Omega(A^\ac), 0\leq k\leq j\},$$
and
$$ \mathrm F_i(\mathrm{CH}^1_\bullet(A;A_\sigma))=\bigoplus_{j\leq i}\{(b_0,b_1,\cdots, b_j)|b_k\in A, 0\leq k\leq j\}.$$
The boundary maps are compatible with the filtrations respectively. Then the comparison theorem for spectral sequences guarantees the quasi-isomorphism. Similarly we can prove other maps are quasi-isomorphisms. 
\end{proof}

\begin{lemma}
Let $A$ be a Koszul algebra with an semi-simple automorphism $\sigma$ and $A^{\ac}$ be its Koszul dual coalgebra. Then we have the following quasi-isomorphisms of mixed complexes
$$
\xymatrix{
&\mathrm{CH}^1_\bullet(\Omega(A^{\ac});\Omega(A^\ac)_\sigma)\ar[dl]_{p_1}\ar[dr]^{q_2}&\\
\mathrm{CH}^1_\bullet(A;A_\sigma)& & \mathrm{CH}^1_\bullet(A^{\ac};\;_\sigma\!A^{\ac})}
$$
where $q_2$ is a homotopy inverse of $p_2$.
\end{lemma}

\begin{proof}
Since $\Omega(A^{\ac})\simeq A$ is a quasi-isomorphism of differential graded algebras, 
the map $p_1$ given in the previous lemma is a quasi-isomorphism of mixed complexes(\cite{Loday}, Proposition 2.5.15). 
We next construct the quasi-isomorphism $q_2$ of the mixed complexes, which is
a homotopy inverse of $p_2$.

From now on, let us denote any homogeneous element in the bar construction $\mathrm B\Omega(A^\ac)$ of $\Omega(A^\ac)$ by
$(a_1,\cdots,a_n)$ with $a_i\in \Omega(A^\ac)$,
and any homogeneous element in the cobar construction $\Omega(A^\ac)$ of $A^\ac$
by $(u_1\cdots u_n)$ or $(v_1\cdots v_m)$. 
The morphism $q_2$ is defined by
\begin{eqnarray}
\mathrm{CH}^1_\bullet(\Omega(A^{\ac});\Omega(A^\ac)_\sigma)&\rightarrow& \mathrm{CH}^1_\bullet(A^{\ac};\;_\sigma\!A^{\ac}),\nonumber\\
((u_1\cdots u_n), 1)&\mapsto &((u_1\cdots u_n), 1)\label{Hoch1}\\
((u_1\cdots u_n), (v_1\cdots v_m))&\mapsto& \sum_i(-1)^{\epsilon_i}\label{Hoch2}
((v_{i+1}\cdots v_mu_1\cdots u_n \sigma(v_1)\cdots \sigma(v_{i-1})), v_i)\\
((u_1\cdots u_n), a_1,\cdots, a_r)&\mapsto& 0,\quad r\geq 2,\label{Hoch3}
\end{eqnarray}
where $\epsilon_i=(|v_{i+1}|+\cdots +|v_m|-m+i)(|u_1|+\cdots +|u_n|-n+|v_1|+\cdots +|v_{i}|-i)$. 

It is direct to check $q_2$ is a morphism of complexes, that is, it commutes with
the Hochschild differential. Now we show $q_2$ commutes with $\mathrm B$. In fact,

(1) For Hochschild chains like the left hand side of \eqref{Hoch1}, we have:
\begin{eqnarray*}
\mathrm{B}\circ q_2((u_1\cdots u_n), 1)=\mathrm{B}((u_1\cdots u_n), 1)=\sum_{i=1}^n(-1)^{\epsilon_i}(u_{i+1}\cdots u_n\sigma(u_1)\cdots\sigma(u_{i-1}), u_i)
\end{eqnarray*}
and
\begin{eqnarray*}
q_2\circ \mathrm{B}((u_1\cdots u_n), 1)=q_2(1, (u_1\cdots u_n))= \sum_{i=1}^n(-1)^{\epsilon_i}(u_{i+1}\cdots u_n\sigma(u_1)\cdots\sigma(u_{i-1}), u_i).
\end{eqnarray*}
This means $q_2\circ \mathrm B=\mathrm B\circ q_2$ in this case.

(2) For Hochschild chains like the left hand side of \eqref{Hoch2}, we have: 
\begin{eqnarray*}
&&\mathrm{B}\circ q_2((u_1\cdots u_n), (v_1\cdots v_m))\\
&=&\mathrm{B}(\sum_{i=1}^m(-1)^{\epsilon_i}v_{i+1}\cdots v_mu_1\cdots u_n\sigma(v_1)\cdots \sigma(v_{i-1}), v_i)\\
&=&0
\end{eqnarray*}
and
\begin{eqnarray*}
&&q_2\circ \mathrm{B}((u_1\cdots u_n), (v_1\cdots v_m))=q_2(1, M)=0,
\end{eqnarray*}
for some $M\in \Omega(A^{\ac})^{\otimes 2}$.
This means $q_2\circ \mathrm B=\mathrm B\circ q_2$ in this case.

(3) For Hochschild chains like the left hand side of \eqref{Hoch3}, $q_2\circ \mathrm B=\mathrm B\circ q_2$ is automatic since both sides are always zero.

In summary, the above calculation implies that $q_2$ is a morphism of mixed complexes.
Next, we show that $p_2$ and $q_2$ are homotopy inverse to each other.

First, 
since
$$q_2\circ p_2((u_1\cdots u_n), v)=q_2((u_1\cdots u_n), (v+\Delta(v)+\cdots))=(u_1\cdots u_n), v),$$
we get $q_2\circ p_2=id$ on $\Omega(A^{\ac})_{\sigma}\otimes A^{\ac}$.

Second, we show
$$p_2\circ q_2\simeq id: \Omega(A^{\ac})_{\sigma}\otimes \mathrm{B}\Omega(A^{\ac})\rightarrow \Omega(A^{\ac})_{\sigma}\otimes \mathrm{B}\Omega(A^{\ac}).$$
The homotopy map, denoted by $h$, is given as follows:
First, let $$h_0(a)=0, \quad \mbox{for}\; a\in \Omega(A^{\ac}),$$
and for $n\geq 1$, 
$$h_n(a_0, a_1,\cdots, a_n, v)=0,$$
and
\begin{eqnarray*}
&&h_n(a_0, a_1,\cdots, a_{n-1}, a_nv)\\
&=&(-1)^{\mu_n}h_n(va_0, a_1,\cdots, a_{n-1}, a_n)+(-1)^{\nu_n}h_{n+1}(a_0, a_1,\cdots, a_n, v'v'')\\
&&+(-1)^{\epsilon_n}(a_0, a_1, \cdots, a_n, v),
\end{eqnarray*}
where $a_i\in \Omega(A^{\ac})$ for $i=0,\cdots,n$ and $v\in A^{\ac}$.
Here we use the Sweedler notation $\Delta(v)=\sum_{(v)}v_{(1)}\otimes v_{(2)}:=v'\otimes v''$. The signs are given by
$$\mu_n=(|v|-1)(|a_0|+\cdots +|a_n|+n),\quad \nu_n=|a_0|+\cdots +|a_n|+n+|v'|,\quad \epsilon_n=|a_0|+\cdots +|a_{n-1}|+n-1,$$
where $|a_i|=|a_i^1|+\cdots +|a_i^{i_s}|-i_s$, $i=0,\cdots, n$ for any $a_i=(a_i^1\cdots a_i^{i_s})$ with $a_i^j\in A^\ac$, $j=1,\cdots, i_s$.
Now let $h=\sum_{i=0}h_i$ and we claim that
\begin{equation}\label{oneidentity}
h\circ d+d\circ h=id-p_2\circ q_2.
\end{equation}
Assuming this identity, we obtain that $q_2$ is a quasi-isomorphism of chain complexes, and thus a quasi-isomorphism of mixed complexes. Then the proof is completed.
\end{proof}

\begin{proof}[Proof of \eqref{oneidentity}]
For any element $(a_0, \cdots, a_m,v_1\cdots v_n)\in \Omega(A^{\ac})_\sigma\otimes \mathrm{B}\Omega(A^{\ac})$, 
where $a_i$ and $(v_1\cdots v_m)\in\Omega(A^{\ac})$,
and $v_i\in A^\ac$ for $i=1,\cdots, n$,
we have
\begin{eqnarray*}
&&h(a_0,\cdots, a_n, v_1\cdots v_m)\\
&=&\sum_{i=2}^m\sum_{k=0}^\infty(-1)^{i_k}(v_{i+1}\cdots v_m a_0, \cdots, a_n, v_1\cdots v_{i-1}, \Delta^k(v_i)),
\end{eqnarray*}
where 
\begin{eqnarray*}
i_k&=&\mu_m+\cdots+\mu_{i+1}+|a_0|+\cdots+|a_n|+n+|v_1|+\cdots+|v_{i-1}|-i+1\\
&&+(k-1)|v_i^{(1)}|+\cdots+(k-i)|v_i^{(i)}|+\cdots +|v_i^{(k-1)}|,
\end{eqnarray*}
and
$$
\mu_s=(|v_{s}|-1)(|a_0|+\cdots+|a_n|+n+|v_1|+\cdots+|v_m|-|v_{s}|-m+1),
$$
for $s=i+1,\cdots, m$,
and
where we write $\Delta^k(v)=v^{(1)}\otimes\cdots \otimes v^{(k+1)}$. We have the following three
cases to check $h\circ d+d\circ h=id-p_2\circ q_2$:

(1) For Hochschild chains like the left hand side of \eqref{Hoch1}, it is direct to see 
$$(h\circ d+d\circ h)(a_0, 1)=(id-p_2\circ q_2)(a_0, 1)=0.$$

(2) For Hochschild chains like the left hand side of \eqref{Hoch2},
we have $$d\circ h(a,wv)= \sum_{n=0}(-1)^{\epsilon_n} d(a,w,\Delta^n(v))$$ 
with $\epsilon_n=|a|+|w|+1+(n-1)|v^{(1)}|+\cdots +(n-k)|v^{(k)}|+\cdots +|v^{(n-1)}|$, where $a=(u_1\cdots u_m)$ and $|a|=|u_1|+\cdots+|u_m|-m$, is equal to 
\begin{eqnarray*}
&&\sum_{n=0}(-1)^{\epsilon_n}\big(-(-1)^{|a|}(a\sigma(w),\Delta^n(v))-(-1)^{|a|+|\omega|+1}(a,wv^{(1)},v^{(2)},\cdots, v^{(n)})\\
&&-\sum_{s=1}^{n-1}(-1)^{|a|+|\omega|+|v^{(1)}|+\cdots+|v^{(s)}|+s+1}(a,w,v^{(1)},\cdots, v^{(s)}v^{(s+1)},\cdots, v^{(n)})\\
&&+(-1)^{|v^{(n)}-1|(|a|+|w|+|v^{(1)}+\cdots+|v^{(n-1)}|+n)}(v^{(n)}a, w,v^{(1)},\cdots, v^{(n-1)})\\
&&+\sum_{i=1}^m(-1)^{|u_1|+\cdots+|u_{i-1}|-i+1+|u_i'|+1}((u_1\cdots u_i'u_i''\cdots u_m),w,\Delta^n(v))\\
&&+(-1)^{|a|+|w'|+1}(a, w'w'',\Delta^n(v))\\
&&+\sum_{i=1}^n(-1)^{|a|+|w|+|v^{(1)}|+\cdots +|v^{(i-1)}|+i+|v^{(i)'}|+1}(a,w,v^{(1)},\cdots, v^{(i)'}v^{(i)''},\cdots, v^{(n-1)})\big).
\end{eqnarray*}
The second summand 
\begin{eqnarray*}
h\circ d(a,wv)&=&\sum_{i=1}^m (-1)^{|u_1|+\cdots +|u_{i-1}|-i+|u_i'|}h((u_1\cdots u_i'u_i''\cdots u_m),wv)\\
&&+h(a,(-1)^{|a|+|w'|+1}w'w''v+(-1)^{|a|+|w|+|v'|+1}wv'v'')
\end{eqnarray*}
has the following terms (we omit the sign)
\begin{eqnarray*}
&&(\Delta(a), w,v\pm\Delta(v)+\cdots)\pm(va, \Delta(w)\pm\cdots)\pm (a,\Delta(w), v\pm\Delta(v)\pm\cdots)\\
&&\pm(v''a,w,v'\pm\Delta(v')\pm\cdots)\pm(a,wv',v''\pm\Delta(v'')\pm\cdots).
\end{eqnarray*}
Recall the Construction of bar and cobar construction and the corresponding differentials, we obtain that $h\circ d+d\circ d$ is equal to
\begin{eqnarray*}
(id-p_2\circ q_2)(a,wv)&=& (a,wv)\pm p_2(a\sigma(w),v)\pm p_2(va,w)\\
&=&(a,wv)\pm (a\sigma(w),v\pm \Delta(v)\pm \cdots)\pm (va,w\pm \Delta(w)\pm \cdots).
\end{eqnarray*}
The above identities implies that 
$$d\circ h+h\circ d=id-p_2\circ q_2$$
in this case. For elements $(a,u_1\cdots u_m)\in \Omega(A^\ac)_\sigma\otimes \mathrm B\Omega(A^\ac)$ with $m\geq 3$, it is direct to obtain the identity
$$d\circ h+h\circ d=id-p_2\circ q_2$$
by similar computations as above.

(3) For Hochschild chains like the left hand side of \eqref{Hoch3} like
$$(a_0, \cdots, a_k,u_1\cdots u_m,v_1\cdots v_n)\in \Omega(A^{\ac})_\sigma\otimes \mathrm{B}\Omega(A^{\ac}),$$ 
where $a_i$, $(u_1\cdots u_m)$ and $(v_1\cdots v_n)\in\Omega(A^{\ac})$,
and $u_i, v_j\in A^\ac$ for $i=1,\cdots, m, j=1,\cdots, n$, we have that $d\circ h(a_0, \cdots, a_k,u_1\cdots u_m,v_1\cdots v_n)$ equals
\begin{eqnarray*}
&&(v_{i+1}\cdots\Delta(v_j)\cdots v_m a_0, a_1, \cdots, a_k, u_1\cdots u_m, v_1\cdots v_{i-1}, \Delta^s(v_i))\\
&&\pm(v_{i+1}\cdots v_m \Delta(a_0), a_1,\cdots, a_k, u_1\cdots u_m, v_1\cdots v_{i-1}, \Delta^s(v_i))\\
&&\pm(v_{i+1}\cdots v_m a_0, a_1, \cdots,\Delta(a_j),\cdots, a_k, u_1\cdots u_m, v_1\cdots v_{i-1}, \Delta^s(v_i))\\
&&\pm(v_{i+1}\cdots v_m a_0, a_1, \cdots, a_k, u_1\cdots\Delta(u_j)\cdots u_m, v_1\cdots v_{i-1}, \Delta^s(v_i))\\
&&\pm(v_{i+1}\cdots v_m a_0, a_1, \cdots, a_k, u_1\cdots u_m, v_1\cdots\Delta(v_j), v_{i-1}, \Delta^s(v_i))\\
&&\pm(v_{i+1}\cdots v_m a_0, a_1, \cdots, a_k, u_1\cdots u_m, v_1\cdots v_{i-1}, v_i^{(1)},\cdots, v_i^{(j)'}v_i^{(j)''},\cdots, v_i^{(s+1)})\\
&&\pm(v_{i+1}\cdots v_m a_0\sigma(a_1), a_2,\cdots, a_k, u_1\cdots u_m, v_1\cdots v_{i-1}, \Delta^s(v_i))\\
&&\pm(v_{i+1}\cdots v_m a_0, a_1, \cdots, a_ja_{j+1},\cdots, a_k, u_1\cdots u_m, v_1\cdots v_{i-1}, \Delta^s(v_i))\\
&&\pm(v_{i+1}\cdots v_m a_0, a_1, \cdots, a_{k-1} a_ku_1\cdots u_m, v_1\cdots v_{i-1}, \Delta^s(v_i))\\
&&\pm(v_{i+1}\cdots v_m a_0, a_1, \cdots, a_k, u_1\cdots u_mv_1\cdots v_{i-1}, \Delta^s(v_i))\\
&&\pm(v_{i+1}\cdots v_m a_0, a_1, \cdots, a_k, u_1\cdots u_m, v_1\cdots v_{i-1}v_i^{(1)},v_i^{(2)},\cdots, v_i^{(s+1)})\\
&&\pm(v_{i+1}\cdots v_m a_0, a_1, \cdots, a_k, u_1\cdots u_m, v_1\cdots v_{i-1}, v_i^{(1)},\cdots, v_i^{(j)}v_i^{(j+1)},\cdots, v_i^{(s+1)})\\
&&\pm(v_i^{(s+1)}v_{i+1}\cdots v_m a_0, a_1, \cdots, a_k, u_1\cdots u_m, v_1\cdots v_{i-1}, v_i^{(1)},\cdots, v_i^{(s)}),
\end{eqnarray*}
and the second summand $h\circ d(a_0,a_1, \cdots, a_k,u_1\cdots u_m,v_1\cdots v_n)$ equals
\begin{eqnarray*}
&&(v_{i+1}\cdots v_m a_0\sigma(a_1), a_2,\cdots, a_k, u_1\cdots u_m, v_1\cdots v_{i-1}, \Delta^s(v_i))\\
&&\pm(v_{i+1}\cdots v_m a_0, a_1, \cdots, a_ja_{j+1},\cdots, a_k, u_1\cdots u_m, v_1\cdots v_{i-1}, \Delta^s(v_i))\\
&&\pm(v_{i+1}\cdots v_m a_0, a_1, \cdots, a_{k-1} a_ku_1\cdots u_m, v_1\cdots v_{i-1}, \Delta^s(v_i))\\
&&\pm(v_{i+1}\cdots v_m a_0, a_1, \cdots, a_k, u_1\cdots u_mv_1\cdots v_{i-1}, \Delta^s(v_i))\\
&&\pm(u_{i+1}\cdots u_mv_1\cdots v_n a_0,a_1,\cdots, a_k, u_1\cdots u_{i-1},\Delta^s(u_i))\\
&&\pm(u_{i+1}\cdots u_mv_1\cdots v_n a_0,a_1,\cdots, a_k, u_1\cdots u_{i-1},\Delta^s(u_i))\\
&&\pm(v_{i+1}\cdots v_m \Delta(a_0), a_1,\cdots, a_k, u_1\cdots u_m, v_1\cdots v_{i-1}, \Delta^s(v_i))\\
&&\pm(v_{i+1}\cdots v_m a_0, a_1, \cdots,\Delta(a_j),\cdots, a_k, u_1\cdots u_m, v_1\cdots v_{i-1}, \Delta^s(v_i))\\
&&\pm(v_{i+1}\cdots v_m a_0, a_1, \cdots, a_k, u_1\cdots\Delta(u_j)\cdots u_m, v_1\cdots v_{i-1}, \Delta^s(v_i))\\
&&\pm(v_{i+1}\cdots\Delta(v_j)\cdots v_m a_0, a_1, \cdots, a_k, u_1\cdots u_m, v_1\cdots v_{i-1}, \Delta^s(v_i))\\
&&\pm(v_i''v_{i+1}\cdots v_m a_0, a_1, \cdots, a_k, u_1\cdots u_m, v_1\cdots v_{i-1}, \Delta^s(v_i'))\\
&&\pm(v_{i+1}\cdots v_m a_0, a_1, \cdots, a_k, u_1\cdots u_m, v_1\cdots v_{i-1}v_i', \Delta^{s-1}(v_i''))\\
&&\pm(v_{i+1}\cdots v_m a_0, a_1, \cdots, a_k, u_1\cdots u_m, v_1\cdots\Delta(v_j)\cdots v_{i-1}, \Delta^{s-1}(v_i)),
\end{eqnarray*}
where $\Delta(a_i)=\sum_{s=1}^{n_i}(a_i^1\cdots a_i^{s-1}a_i^{s'}a_i^{s''}a_i^{s+1}\cdots a_i^{n_i})$ if $a_i$ has the form $a_i=(a_i^1\cdots a_i^{n_i})$ with $a_i^j\in A^\ac$, $j=0,\cdots, n_i$.
From the coassociativity of the coproduct $\Delta$ on $A^\ac$, we see that
the sum of the above two expressions
$$(d\circ h+h\circ d)(a_0,a_1,\cdots, a_k, u_1\cdots u_m, v_1\cdots v_n)=(a_0,a_1,\cdots, a_k, u_1\cdots u_m, v_1\cdots v_n).$$
Since $q_2(a_0,a_1,\cdots, a_k, u_1\cdots u_m, v_1\cdots v_n)=0$ for $k\geq 1$, we have the identity
$$d\circ h+ h\circ d=id-p_2\circ q_2.$$

In summary, we proved the desired identity.
\end{proof}

\begin{theorem}\label{thm:Koszuldualitygivesisoofhomology}
Suppose $A$ is a Koszul AS-regular algebra with semisimple Nakayama automorphism $\sigma$.
Then we have isomorphism
$$
\mathrm{HH}_\bullet(A; A_\sigma)\cong \mathrm H_\bullet(\mathrm K_\bullet(A_\sigma))
\cong\mathrm{HH}_\bullet(A^{\ac}; \; _\sigma\!A^{\ac}),
$$
and such isomorphisms respect the Connes operator on both sides.
\end{theorem}

\begin{proof}
Combining the above two lemmas and the following fact
$$(A_\sigma\otimes A^\ac,b)\cong (A\otimes\;_\sigma\!A^\ac,b)\simeq (\Omega(A^\ac)\otimes \;_\sigma\! A^\ac,\delta^*)$$
where we use the $A$-bimodule structure of $A_\sigma$ in $A_\sigma\otimes A^\ac$ and 
the $A^\ac$-cobimodule structure of $\;_\sigma\!A^\ac$ in $A\otimes \;_\sigma\!A^\ac$, we get the proof.
\end{proof}

Now consider the following complex
$$\mathrm K^\bullet(A\otimes A^!)=\{
0\rightarrow A\otimes A\rightarrow \cdots\rightarrow A\otimes A^!_{n-1}\otimes A\rightarrow A\otimes A^!_n\otimes A\rightarrow \cdots\}$$
with coboundary $\delta$ given by
$$\delta(a\otimes f\otimes b):=a\otimes x_i^*f\otimes x_ib+(-1)^{m}ax_i\otimes fx_i^*\otimes b.$$

\begin{theorem}\label{thm:Koszuldualitygivesisoofcohomology}
Let $A$ be a Koszul algebra and denoted its Koszul dual algebra by $A^!$. Then there are natural isomorphisms
$$\mathrm{HH}^\bullet(A)\cong \mathrm{H}^\bullet(\mathrm K^\bullet(A\otimes A^!))\cong \mathrm{HH}^\bullet(A^!)$$
of graded commutative algebras. The products on both sides are the Gerstenhaber cup product.
\end{theorem}

\begin{proof}
See Buchweitz \cite{Buchweitz}, or Beilinson-Ginzburg-Soergel \cite{BGS}, or Keller \cite{Keller}.
\end{proof}

\subsection{Proof of the main theorem}

We are now ready to prove the main theorem of this paper.

\begin{proof}[Proof of Theorem \ref{mainthm}]
By Theorems \ref{PDforASregular} and \ref{PDforFrobenius}
we have the following commutative diagram
$$
\xymatrix{
\mathrm{HH}^\bullet(A; A)\ar[rr]&&\mathrm{HH}_{n-\bullet}(A; A_\sigma)&\\
&\mathrm{H}^\bullet(\mathrm K^\bullet(A\otimes A^!))
\ar[rr]\ar[lu]\ar[ld]&&\mathrm H_{n-\bullet}(\mathrm
K_\bullet(A\otimes \;_\sigma\!A^{\ac}))\ar[lu]\ar[ld]\\
\mathrm{HH}^\bullet(A^!; A^!)\ar[rr]
&&\mathrm{HH}_{n-\bullet}(A^{\ac};\;_\sigma\!A^{\ac}),&
}
$$
which gives the following commutative diagram
$$
\xymatrix{
\mathrm{HH}^\bullet(A; A)\ar[r]\ar[d]&\mathrm{HH}_{n-\bullet}(A; A_\sigma)\ar[d]\\
\mathrm{HH}^\bullet(A^!; A^!)\ar[r]&\mathrm{HH}_{n-\bullet}(A^{\ac};\;_\sigma\!A^{\ac}).
}
$$
Theorem \ref{thm:Koszuldualitygivesisoofcohomology} 
says that the left vertical map is an isomorphism of graded algebras, and
Theorem \ref{thm:Koszuldualitygivesisoofhomology} says that the right vertical map
respects the Connes operators.
Thus by Lemma \ref{calwithdualityinducesBV}
we see that two Batalin-Vilkovisky algebras are isomorphic.
\end{proof}

\begin{ack}I would like to
thank Xiaojun Chen and Farkhod Eshmatov for many helpful conversations and encouragements.
This work is partially supported by NSFC (No. 11521061 and 11671281).
\end{ack}

\end{document}